\documentclass[10pt]{article}
\usepackage{amsfonts,amssymb,amsbsy,latexsym,amsthm,amsmath,tabulary,graphicx,times,caption,fancyhdr}
\usepackage[utf8]{inputenc}
\usepackage{amscd, graphicx, color}
\usepackage{cite}
\usepackage{url,multirow,morefloats,floatflt,cancel,tfrupee}

\usepackage{enumerate}
\usepackage{enumitem}
\usepackage{authblk}
\usepackage{cases}
\usepackage[bookmarksnumbered, plainpages]{hyperref}
\usepackage{fancyhdr}
\usepackage{caption}
\usepackage{subcaption}
\usepackage{cite}
\usepackage[a4paper,bindingoffset=0.2in,%
left=1in,right=1in,top=1in,bottom=1in,%
footskip=0.30in]{geometry}
\usepackage{etoolbox}

\usepackage{xargs} 
\usepackage{bbm} 
\usepackage{stackrel} 
\usepackage{xspace}  
\usepackage{diagbox} 
\newtheorem{theorem}{Theorem}[section]

\newtheorem{corollary}[theorem]{Corollary}
\newtheorem{proposition}[theorem]{Proposition}
\newtheorem{notation}[theorem]{Notation}

\theoremstyle{definition}
\newtheorem{remark}[theorem]{Remark}
\newtheorem{definition}[theorem]{Definition}
\newtheorem{example}[theorem]{Example}

\newcommand{\df}[1]{\textbf{\textit{#1}}}  
\newcommand{\U}{\mathbbmss{U}}
\newcommand{\V}{\mathbbmss{V}}
\newcommand{\RR}{\mathbb{R}} 
\newcommand{\NN}{\mathbb{N}} 
\newcommand{\cF}{\mathcal{F}}
\newcommand{\cG}{\mathcal{G}}
\newcommand{\cH}{\mathcal{H}}
\newcommand{\cU}{\mathcal{U}}
\newcommand{\cS}{\mathcal{S}}
\newcommand{\cA}{\mathcal{A}}
\newcommand{\cB}{\mathcal{B}}
\newcommand{\cM}{\mathcal{M}}
\newcommand{\mF}{\mathfrak{F}}  
\newcommand{\mS}{\mathfrak{S}}  
\newcommand{\mC}{\mathfrak{C}}  
\newcommand{\und}{\underline}  

\newcommand{\M}[1]{\mu_{#1}}       
\newcommand{\I}[1]{\sigma_{#1}}       
\newcommand{\NM}[1]{\omega_{#1}}   
\newcommand{\DM}[2][u]{\mu_{#2}\left(#1\right)}       
\newcommand{\DI}[2][u]{\sigma_{#2}\left(#1\right)}       
\newcommand{\DNM}[2][u]{\omega_{#2}\left(#1\right)}   
\newcommand{\ANG}[1]{\left\langle #1 \right\rangle}  
\newcommand{\NSbase}[4][\U]{\ANG{#1 , #2, #3, #4}}  
\newcommand{\ns}[1]{\widetilde{#1}}   
\newcommand{\NS}[2][\U]{\NSbase[#1]{\M{#2}}{\I{#2}}{\NM{#2}}}   
\newcommand{\nNS}[2][\U]{\ns{#2} = \NS[#1]{#2}}  
\newcommandx{\NSEXT}[3][1=u,2=\U]{\left\{ \left( #1, \DM{#3}, \DI{#3}, \DNM{#3} \right) : \, #1 \in #2 \right\}} 
\newcommand{\nameSVNS}[1][]{SVN-set#1\xspace}  
\newcommand{\SSVNS}[1][\U]{\mathcal{SVN}(#1)}  
\newcommand{\completion}[1]{\ANG{#1}}   
\newcommand{\SVNF}[1][\U]{\mF(#1)}   
\newcommand{\NSemptyset}[1][]{\widetilde{\emptyset}_{#1}}
\newcommand{\NSabsoluteset}[1][\U]{\widetilde{#1}}
\newcommand{\NSsubseteq}{\mathrel{\ooalign{
\raise0.2ex\hbox{$\Subset$}%
\cr\hidewidth%
\raise-0.25ex\hbox{\rule[0pt]{5.6pt}{0.4pt}}%
\hidewidth\cr%
}}}
\newcommand{\NSnotsubseteq}{\not\NSsubseteq} 
\newcommand{\NSeq}{\mathrel{\ooalign{
\raise0.1ex\hbox{$=$}%
\cr%
\hskip0.6pt\raise0.2ex\hbox{\rule{6.6pt}{0.35pt}}
\cr\hidewidth%
\raise1.05ex\hbox{{\rule{6.6pt}{0.32pt}}}
\hidewidth\cr%
}}}
\newcommand{\NSneq}{\not\NSeq}   
\newcommand{\NScompl}{{\hskip.15ex\ooalign{\hbox{\scalebox{0.8}{$\complement$}}%
\hidewidth\cr\hspace{1.6pt}\hbox{\rule[0.2pt]{0.6pt}{6.6pt}}%
}}}    
\newcommand{\NScomplsmall}{{\hskip.05ex\ooalign{\hbox{\scalebox{0.6}{$\complement$}}%
\hidewidth\cr\hspace{1.2pt}\hbox{\rule[0.2pt]{0.4pt}{4.6pt}}%
}}}    
\newcommand{\NScup}{\Cup}  
\newcommand{\NScap}{\Cap}  
\newcommand{\NSCup}{
\mathop{\vphantom{\bigcup}\vcenter{\hbox{\text{%
\ooalign{$\displaystyle\bigcup$\cr%
\hidewidth%
\raisebox{.25ex}{\resizebox{.7\width}{.875\height}{$\displaystyle\bigcup$}}%
\hidewidth\cr
}}}}}}%
\newcommand{\NSCap}{
\mathop{\vphantom{\bigcap}\vcenter{\hbox{\text{%
\ooalign{$\displaystyle\bigcap$\cr%
\hidewidth%
\raisebox{-.1ex}{\resizebox{.7\width}{.875\height}{$\displaystyle\bigcap$}}%
\hidewidth\cr%
}}}}}}%
\newcommand{\inv}[1]{{#1}^{-1}\!}  
\newcommand{\fibre}[2]{{#1}^{-1}\left(\left\{#2\right\}\right)}    

\newenvironment{enumi}[1][0]%
{\begin{enumerate}[label={\rm(\arabic*)},topsep=4pt]%
\itemsep1pt
\setcounter{enumi}{#1}%
}%
{\end{enumerate}}

{\begin{enumerate}[label={\rm(\alph*)},topsep=4pt]%
\itemsep1pt
\setcounter{enumi}{#1}%
}%
{\end{enumerate}}

\newenvironment{enumr}[1][0]%
{\begin{enumerate}[label={\rm(\roman*)},topsep=4pt]%
\itemsep1pt
\setcounter{enumi}{#1}%
}%
{\end{enumerate}}

\newenvironment{nstabular}[5]%
{\begin{tabular}[t]{|c|c|c|c|}\hline
\diagbox{\quad$#1$}{\\[1mm]\quad$#2$} &
$#3$ & $#4$ & $#5$ \\ \hline}%
{\\ \hline \end{tabular}}

\begin{document}

\begin{center}
\textbf{\Large Single Valued Neutrosophic Filters}\\ \vspace{10pt}
$^1$Giorgio Nordo, $^2$Arif Mehmood, $^3$Said Broumi\\
$^1$MIFT - Department of Mathematical and Computer Science, Physical Sciences and Earth Sciences,\\
Messina University, Italy.
\\
$^2$Department of Mathematics and Statistics, Riphah International University Sector I-14, Islamabad, Pakistan
\\
$^3$Laboratory of Information Processing, Faculty of Science, University Hassan II, Casablanca, Morocco
\\ \vspace{4pt}
giorgio.nordo@unime.it$^1$, mehdaniyal@gmail.com$^2$, broumisaid78@gmail.com$^3$
\end{center}

\section*{Abstract}
In this paper we give a comprehensive presentation of the notions of filter base,
filter and ultrafilter on single valued neutrosophic set
and we investigate some of their properties and relationships.
More precisely, we discuss properties related to filter completion, the image of neutrosophic filter base by a
neutrosophic induced mapping and the infimum and supremum of two neutrosophic filter bases.
\\
\noindent \textbf{Keywords:} neutrosophic set, single valued neutrosphic set, neutrosophic induced mapping,
single valued neutrosophic filter, neutrosophic completion, single valued neutrosophic ultrafilter.


\section{Introduction}
The notion of neutrosophic set was introduced in 1999 by Smarandache \cite{smarandache}
as a generalization of both the notions of fuzzy set introduced by Zadeh in 1965 \cite{zadeh} and
intuitionistic fuzzy set introduced by Atanassov in 1983 \cite{atanassov}.

In 2012, Salama and Alblowi \cite{salama2012} introduced the notion of neutrosophic topological space
which generalizes both fuzzy topological spaces given by Chang \cite{chang}
and that of intuitionistic fuzzy topological spaces given by Coker \cite{coker}.
Further contributions to Neutrosophic Sets Theory,
which also involve many fields of theoretical and applied mathematics
were recently given by numerous authors (see, for example,
\cite{broumi}, \cite{bera2017}, \cite{bera2018}, \cite{parimala2020},
\cite{al-omeri}, \cite{mehmood2019}, \cite{mehmood2020} and \cite{saber2020}).
In particular, in \cite{salama2013} Salama and Alagamy introduced and studied the notion of neutrosophic filter
and they gave some applications to neutrosophic topological space.

In General Topology, filter bases, filters and ultrafilters are
widely known notions and very popular tools for proving
many properties and characterizations (see, for example \cite{bourbaki,cartan,engelking}).
\\
Rather surprisingly, despite the fact that the class of single valued neutrosophic sets,
is more versatile and has a particular aptitude
for application purposes and resolution of practical real-world problems
than that of neutrosophic sets,
the authors of this article were not able to find
any generalizations of such notions respect on single valued neutrosophic sets,
in known scientific literature.

In this paper, we introduce the notions of filter base, filter and ultrafilter
on single valued neutrosophic sets,
and we prove some of their fundamental properties and relationships
which may be useful for further studies and applications in the class of
single valued neutrosophic topological spaces.

\section{Preliminaries}
In this section we present some basic definitions and results on neutrosophic sets and suitably exemplify them.
Terms and undefined concepts are used as in \cite{engelking} and \cite{gemignani}.

The original definition of neutrosophic set, given in 1999 by Smarandache \cite{smarandache},
refers to the interval $]0^{-},1^{+}[$ of the nonstandard real numbers
and although it is consistent from a philosophical point of view,
unfortunately, it is not suitable to be used for approaching real-world problems.
For such a reason, in 2010, the same author, jointly with Wang, Zhang and Sunderraman \cite{wang},
also introduced the notion of single valued neutrosophic set which,
referring instead to the $[0,1]$ unit range of the usual $\RR$ set of real numbers,
can be usefully used in scientific and engineering applications.

\begin{notation}
Let $\U$ be a set, $I=[0,1]$ be the unit interval of the real numbers, for every $r \in I$,
with $\und{r}$ we denote the constant mapping $\und{r}: \U \to I$ that,
for every $u \in \U$ is defined by $\und{r}(u) = r$.
\\
For every family $\left\{ f_i \right\}_{i \in I}$ of mappings $f_i : \U \to I$,
we denote by:
\begin{itemize}
\item $\bigwedge_{i \in I} f_i$ the infimum mapping $\bigwedge_{i \in I} f_i : \U \to I$
that, for every $u \in\ U$ is defined by $\left( \bigwedge_{i \in I} f_i \right) (u) = \bigwedge_{i \in I} f_i(u)
= \inf \left\{ f_i(u): i \in I \right\}$, and by

\item $\bigvee_{i \in I} f_i$ the supremum mapping $\bigvee_{i \in I} f_i : \U \to I$
that, for every $u \in\ U$ is defined by $\left( \bigvee_{i \in I} f_i \right) (u) = \bigvee_{i \in I} f_i(u)
= \sup \left\{ f_i(u): i \in I \right\}$.
\end{itemize}
In particular, if $f$ and $g$ are two mappings from $\U$ to $I$,
we denote their infimum (which is the minimum) by $f \wedge g$
and their supremum (which is the maximum) by $f \vee g$.
\end{notation}

\begin{definition}{\rm\cite{wang}}
\label{def:singlevaluedneutrosophicset}
Let $\U$ be an initial universe set and $A\subseteq \U$,
a \df{single valued neutrosophic set} over $\U$ (\df{\nameSVNS} for short),
denoted by $\nNS{A}$, is a set of the form
$$\ns{A} = \NSEXT{A}$$  
where $\M{A} : \U \to I$, $\I{A} : \U \to I$ and $\NM{A} : \U \to I$
are the \df{membership function}, the \df{indeterminacy function} and the \df{nonmembership function} of $A$,
respectively.
For every $u \in \U$, $\DM{A}$, $\DI{A}$ and $\DNM{A}$ are said
the \df{degree of membership}, the \df{degree of indeterminacy} and the \df{degree of nonmembership}
of $u$, respectively.
\end{definition}

Since $I=[0,1]$, it clearly results $0 \le \DM{A} + \DI{A} + \DNM{A} \le 3$,
for every $u \in \U$.

\begin{notation}
\label{not:setofsinglevaluedneutrosophicsets}
The set of all the single valued neutrosophic sets over a universe $\U$ will be denoted by $\SSVNS$.
\end{notation}

\begin{definition}{\rm\cite{smarandache,wang}}
\label{def:neutrosophicsubset}
Let $\nNS{A}$ and $\nNS{B}$ be two \nameSVNS[s] over the universe set $\U$,
we say that $\ns{A}$ is a \df{neutrosophic subset} (or simply a subset) of $\ns{B}$
and we write $\ns{A} \NSsubseteq \ns{B}$
if, for every $u \in \U$, it results
$\DM{A} \le \DM{B}$,
$\DI{A} \le \DI{B}$    
and $\DNM{A} \ge \DNM{B}$.
We also say that $\ns{A}$ is contained in $\ns{B}$ or that $\ns{B}$ contains $\ns{A}$.
\end{definition}

It is worth noting that the relation $\NSsubseteq$ satisfies the reflexive, antisymmetrical and transitive properties
and so that $\left( \SSVNS, \NSsubseteq \right)$
forms a partial ordered set (poset) but not a totally ordered set (loset)
as shown in the following example.

\begin{example}
\label{ex:setofsinglevaluedneutrosophicsetnottotallyordered}
Let $\U = \left\{ a,b,c \right\}$ be a finite universe set and
$\nNS{A}, \nNS{B}$ be two \nameSVNS[s] on $\SSVNS$ respectively defined
by the following tabular representations:
\vspace{-3mm}   
\begin{center}
\begin{tabular}{cc}
\begin{nstabular}{\U}{\ns{A}}{\M{A}}{\I{A}}{\NM{A}}
$a$ & 0.5 & 0.3 & 0.2
\\ \hline
$b$ & 0.6 & 0.2 & 0.3
\\ \hline
$c$ & 0.4 & 0.2 & 0.7
\end{nstabular}
\hspace{12mm} &
\begin{nstabular}{\U}{\ns{B}}{\M{B}}{\I{B}}{\NM{B}}
$a$ & 0.2 & 0.2 & 0.2
\\ \hline
$b$ & 0.4 & 0.1 & 0.6
\\ \hline
$c$ & 0.8 & 0.3 & 0.1
\end{nstabular}
\end{tabular}
\end{center}
Then $\ns{A} \NSnotsubseteq \ns{B}$ because $\DM[a]{A} = 0.5 > 0.2 = \DM[a]{B}$
and $\ns{B} \NSnotsubseteq \ns{A}$ because $\DNM[c]{B} = 0.1 < 0.7 = \DNM[c]{A}$
and so the \nameSVNS[s] $\ns{A}$ and $\ns{B}$ are not comparable.
\end{example}

\begin{definition}{\rm\cite{smarandache,wang}}
\label{def:neutrosophicequality}
Let $\nNS{A}$ and $\nNS{B}$ be two \nameSVNS[s] over the universe set $\U$,
we say that $\ns{A}$ is a \df{neutrosophically equal} (or simply equal) to $\ns{B}$ and we write
$\ns{A} \NSeq \ns{B}$
if $\ns{A} \NSsubseteq \ns{B}$ and $\ns{B} \NSsubseteq \ns{A}$.
\end{definition}

\begin{definition}{\rm\cite{wang}}
\label{def:neutrosophicemptyset}
The \nameSVNS $\NSbase{\und{0}}{\und{0}}{\und{1}}$
is said to be the \df{neutrosophic empty set} over $\U$
and it is denoted by $\NSemptyset$,
or more precisely by $\NSemptyset[\U]$ in case it is necessary
to specify the corresponding universe set.
\end{definition}

\begin{definition}{\rm\cite{wang}}
\label{def:neutrosophicabsoluteset}
The \nameSVNS $\NSbase{\und{1}}{\und{1}}{\und{0}}$
is said to be the \df{neutrosophic absolute set} over $\U$
and it is denoted by $\NSabsoluteset$.
\end{definition}

Evidently, for every $A \in \SSVNS$, it results $\NSemptyset \NSsubseteq A \NSsubseteq \NSabsoluteset$.

\begin{definition}{\rm\cite{smarandache,wang}}
\label{def:neutrosophiccomplement}
Let $\nNS{A}$  be a \nameSVNS over the universe set $\U$,
the \df{neutrosophic complement} (or, simply, the complement) of $\ns{A}$, denoted by $\ns{A}^\NScompl$,
is the \nameSVNS
$\ns{A}^\NScompl = \NSbase{\NM{A}}{\und{1}-\I{A}}{\M{A}}$
that is
$\ns{A}^\NScompl = \left\{ \left( u, \DNM{A}, 1-\DI{A}, \DM{A} \right) : u \in \U \right\}$.
\end{definition}

It is a simple matter to verify that for every $\nNS{A} \in \SSVNS$, it results $\left( \ns{A}^\NScompl \right)^{\!\!\NScompl} \NSeq \ns{A}$
and, in particular, that $\NSabsoluteset^\NScompl \NSeq \NSemptyset$
and $\NSemptyset^\NScompl \NSeq \NSabsoluteset$.

\begin{remark}
\label{rem:intersectionandunionwithcomplement}
It is important to point out that, unlike in the crisp sets theory,
the neutrosophic intersection of a \nameSVNS with its complement is not always the neutrosophic empty set,
and the neutrosophic intersection of a \nameSVNS with its complement is not always the neutrosophic absolute set.
In fact, if we consider the universe set $\U = \left\{ a,b \right\}$
and the \nameSVNS on $\SSVNS$ defined
by the following tabular representations:
\vspace{-3mm}
\begin{center}
\begin{nstabular}{\U}{\ns{A}}{\M{A}}{\I{A}}{\NM{A}}
$a$ & 0.2 & 0.6 & 0.8
\\ \hline
$b$ & 1 & 0.5 & 0
\end{nstabular}
\end{center}
we can easily verify that
the neutrosophic intersection, $\ns{A} \NScap \ns{A}^\NScompl$
and the neutrosophic union, $\ns{A} \NScup \ns{A}^\NScompl$
are, respectively, given by the following tabular representations:
\vspace{-3mm}
\begin{center}
\begin{nstabular}{\U}{\ns{A} \NScap \ns{A}^\NScompl}%
{\M{\ns{A} \NScap \ns{A}^\NScomplsmall}}%
{\I{\ns{A} \NScap \ns{A}^\NScomplsmall}}%
{\NM{\ns{A} \NScap \ns{A}^\NScomplsmall}}
$a$ & 0.2 & 0.4 & 0.8
\\ \hline
$b$ & 0 & 0.5 & 1
\end{nstabular}
\end{center}
and
\vspace{-3mm}
\begin{center}
\begin{nstabular}{\U}{\ns{A} \NScup \ns{A}^\NScompl}%
{\M{\ns{A} \NScup \ns{A}^\NScomplsmall}}%
{\I{\ns{A} \NScup \ns{A}^\NScomplsmall}}%
{\NM{\ns{A} \NScup \ns{A}^\NScomplsmall}}
$a$ & 0.8 & 0.6 & 0.2
\\ \hline
$b$ & 1 & 0.5 & 0
\end{nstabular}
\end{center}
and so that $\ns{A} \NScap \ns{A}^\NScompl \NSneq \NSemptyset$
and $\ns{A} \NScup \ns{A}^\NScompl \NSneq \NSabsoluteset$.
\end{remark}

\begin{proposition}{\rm\cite{wang}}
\label{pro:antimonotonicomplementneutrosophicsets}
For every pair $\nNS{A}$ and $\nNS{B}$ of \nameSVNS[s] in $\SSVNS$, we have
that $\ns{A} \NSsubseteq \ns{B}$ iff $\ns{B}^\NScompl \NSsubseteq \ns{A}^\NScompl$.
\end{proposition}


\begin{definition}{\rm\cite{salama2013}}
\label{def:neutrosophicunion}
Let $\left\{ \ns{A}_i \right\}_{i\in I}$ be a family of \nameSVNS[s]
$\ns{A}_i=\NS{A_i}$ over a common universe set $\U$,
its \df{neutrosophic union} (or simply union), denoted by $\NSCup_{i \in I} \ns{A}_i$,
is the neutrosophic set $\nNS{A}$ where
$\M{A} = \bigvee_{i \in I} \M{A_i}$,
$\I{A} = \bigvee_{i \in I} \I{A_i}$, and
$\NM{A} = \bigwedge_{i \in I} \NM{A_i}$.
\\
In particular, the neutrosophic union of two single \nameSVNS[s]
$\nNS{A}$ and $\nNS{B}$, denoted by $\ns{A} \NScup \ns{B}$, is the neutrosophic set defined by
$\NSbase{\M{A} \vee \M{B}}{\I{A} \vee \I{B}}{\NM{A} \wedge \NM{B}}$.
\end{definition}

\begin{definition}{\rm\cite{salama2013}}
\label{def:neutrosophicintersection}
Let $\left\{ \ns{A}_i \right\}_{i\in I}$ be a family of \nameSVNS[s]
$\ns{A}_i=\NS{A_i}$ over a common universe set $\U$,
its \df{neutrosophic intersection} (or simply intersection), denoted by $\NSCap_{i \in I} \ns{A}_i$,
is the neutrosophic set $\nNS{A}$ where
$\M{A} = \bigwedge_{i \in I} \M{A_i}$,
$\I{A} = \bigwedge_{i \in I} \I{A_i}$, and
$\NM{A} = \bigvee_{i \in I} \NM{A_i}$.
In particular, the neutrosophic intersection of two \nameSVNS[s]
$\nNS{A}$ and $\nNS{B}$, denoted by $\ns{A} \NScap \ns{B}$, is the neutrosophic set defined by
$\NSbase{\M{A} \wedge \M{B}}{\I{A} \wedge \I{B}}{\NM{A} \vee \NM{B}}$.
\end{definition}

\begin{definition}{\rm\cite{wang}}
\label{def:neutrosophicdisjoint}
Let $\nNS{A}$ and $\nNS{B}$ be two \nameSVNS[s] over $\U$,
we say that $\ns{A}$ and $\ns{B}$ are \df{neutrosophically disjoint}
if $\ns{A} \NScap \ns{B} \NSeq \NSemptyset$.
On the contrary, if $\ns{A} \NScap \ns{B} \NSneq \NSemptyset$
we say that $\ns{A}$ \df{neutrosophically meets} $\ns{B}$
(or that $\ns{A}$ and $\ns{B}$ neutrosophically meet each other).
\end{definition}

\begin{definition}{\rm\cite{salama2013}}
\label{def:neutrosophicmeets}
Let $\cA, \cB \subseteq \SSVNS$ be two nonempty families of \nameSVNS[s] over $\U$,
we say that $\cA$ \df{neutrosophically meets} $\cB$ (or that $\cA$ and $\cB$ neutrosophically meet each other)
if every member of $\cA$ neutrosophically meets any member of $\cB$,
that is if for every $\ns{A} \in \cA$ and every $\ns{B} \in \cB$
it results $\ns{A} \NScap \ns{B} \NSneq \NSemptyset$.
\\
In particular, if $\nNS{C}$ is a \nameSVNS over $\U$ which neutrosophically meets each member
of the family $\cA$, we say that $\ns{C}$ neutrosophically meets $\cA$.
\end{definition}

The neutrosophic operators of union, intersection and complement satisfy many relations
similar to those of crisp set theory,
which are summarized in the following propositions.

\begin{proposition}{\rm\cite{wang}}
\label{pro:propertiesunionandintersection}
For every \nameSVNS $\nNS{A} \in \SSVNS$, we have:
\begin{enumi}
\item $\ns{A} \NScup \ns{A} \NSeq \ns{A}$
\item $\ns{A} \NScup \NSemptyset\NSeq \ns{A}$
\item $\ns{A} \NScup \NSabsoluteset \NSeq \NSabsoluteset$
\item $\ns{A} \NScap \ns{A} \NSeq \ns{A}$
\item $\ns{A} \NScap \NSemptyset \NSeq \NSemptyset$
\item $\ns{A} \NScap \NSabsoluteset \NSeq A$
\end{enumi}
\end{proposition}

\begin{proposition}{\rm\cite{wang}}
\label{pro:commutativeneutrosophicsets}
For every pair $\nNS{A}$ and $\nNS{B}$ of \nameSVNS[s] in $\SSVNS$, we have:
\begin{enumi}
\item $\ns{A} \NScup \ns{B} \NSeq \ns{B} \NScup \ns{A}$
\item $\ns{A} \NScap \ns{B} \NSeq \ns{B} \NScap \ns{A}$
\end{enumi}
\end{proposition}

\begin{proposition}{\rm\cite{wang}}
\label{pro:associativeneutrosophicsets}
For every triplet $\nNS{A}$,  $\nNS{B}$ and $\nNS{C}$
of \nameSVNS[s] in $\SSVNS$, we have:
\begin{enumi}
\item $\ns{A} \NScap \left( \ns{B} \NScap \ns{C} \right) \NSeq
  \left( \ns{A} \NScap \ns{B} \right) \NScap \ns{C} $
\item $\ns{A} \NScup \left( \ns{B} \NScup \ns{C} \right) \NSeq
  \left( \ns{A} \NScup \ns{B} \right) \NScup \ns{C} $
\end{enumi}
\end{proposition}


\begin{proposition}{\rm\cite{wang}}
\label{pro:neutrosophicsubset_and_neutrosophicoperators}
Let $\nNS{A}$, $\nNS{B} \in \SSVNS$ be two \nameSVNS[s] over a universe $\U$, then:
\begin{enumi}
\item $\ns{A} \NSsubseteq \ns{B}$ iff $\ns{A} \NScap \ns{B} \NSeq \ns{A}$
\item $\ns{A} \NSsubseteq \ns{B}$ iff $\ns{A} \NScup \ns{B} \NSeq \ns{B}$
\end{enumi}
\end{proposition}


\begin{proposition}{\rm\cite{wang}}
\label{pro:absorptionneutrosophicsets}
For every pair $\nNS{A}$ and $\nNS{B}$ of \nameSVNS[s] in $\SSVNS$, we have:
\begin{enumi}
\item $\ns{A} \NScup \left(\ns{A} \NScap \ns{B} \right) \NSeq \ns{A}$
\item $\ns{A} \NScap \left(\ns{A} \NScup \ns{B} \right) \NSeq \ns{A}$
\end{enumi}
\end{proposition}


\begin{proposition}{\rm\cite{wang}}
\label{pro:monotonic_neutrosophic_operators}
Let $\nNS{A}$,  $\nNS{B}$, $\nNS{C}$ and $\nNS{D}$
be \nameSVNS[s] in $\SSVNS$
such that $\ns{A} \NSsubseteq \ns{B}$ and $\ns{C} \NSsubseteq \ns{D}$, then:
\begin{enumi}
\item $\ns{A} \NScup \ns{C} \NSsubseteq \ns{B} \NScup \ns{D}$
\item $\ns{A} \NScap \ns{C} \NSsubseteq \ns{B} \NScap \ns{D}$
\end{enumi}
\end{proposition}


\begin{proposition}{\rm\cite{wang}}
\label{pro:unionandintersectiongeneralized}
Let $\left\{ \ns{A}_i \right\}_{i\in I}$ be a family of \nameSVNS[s]
$\ns{A}_i=\NS{A_i}$ over a common universe set $\U$,
then, for every $i \in I$, we have that
$\NSCap_{i \in I} \ns{A}_i \, \NSsubseteq \, \ns{A}_i
\, \NSsubseteq \, \NSCup_{i \in I} \ns{A}_i $.
\end{proposition}


\begin{proposition}{\rm\cite{wang}}
\label{pro:generalizeddistributive}
Let respectively $\nNS{A} \in \SSVNS$ be a \nameSVNS
and $\left\{ \ns{B}_i \right\}_{i\in I} \subseteq \SSVNS$  be a family of \nameSVNS[s]
$\ns{B}_i=\NS{B_i}$ over a common universe set $\U$,
then we have:
\begin{enumi}
\item $\ns{A} \NScap \left( \NSCup_{i \in I} \ns{B}_i \right) \NSeq \,
\NSCup_{i \in I} \left( \ns{A} \NScap \ns{B}_i \right)$
\item $\ns{A}\NScup \left( \NSCap_{i \in I} \ns{B}_i \right) \NSeq \,
\NSCap_{i \in I} \left( \ns{A} \NScup \ns{B}_i \right)$
\end{enumi}
\end{proposition}

\begin{proposition}{\rm\cite{wang}}
\label{pro:generalizeddemorganlaws}
Let $\left\{ \ns{A}_i \right\}_{i\in I} \subseteq \SSVNS$ be a family of \nameSVNS[s]
$\ns{A}_i=\NS{A_i}$ over a common universe set $\U$, it results:
\begin{enumi}
\item $\left( \NSCup_{i \in I} \ns{A}_i \right)^\NScompl \NSeq \,
\NSCap_{i \in I} \ns{A}_i^\NScompl $
\item $\left( \NSCap_{i \in I} \ns{A}_i \right)^\NScompl \NSeq \,
\NSCup_{i \in I} \ns{A}_i^\NScompl $
\end{enumi}
\end{proposition}

\begin{definition}{\rm\cite{latreche2020,salama2014}}
\label{def:neutrosophicimage}
Let $f: \U \to \V$ be a mapping between two universe sets $\U$ and $\V$,
and $\nNS{A}$ be a \nameSVNS over $\U$.
The \df{neutrosophic image} of $\ns{A}$ by $f$,
denoted by $\ns{f}\left(\ns{A}\right)$, is the \nameSVNS over $\V$ defined by:
$$\ns{f}\left(\ns{A}\right) = \NSbase[\V]
{f\left(\M{A}\right)}%
{f\left(\I{A}\right)}%
{f\left(\NM{A}\right)}$$
where the mappings $f\left(\M{A}\right) : \V \to I$, $f\left(\I{A}\right) : \V \to I$
and $f\left(\NM{A}\right) : \V \to I$ are defined respectively by:
$$
\arraycolsep=1.5pt   
\begin{array}{ll}
f\left(\M{A}\right)(v) &= \left\{%
\arraycolsep=14pt   
\begin{array}{ll}
{\displaystyle \inf_{u\in \fibre{f}{v}} \DM{A} }  &
\text{if } \fibre{f}{v} \neq \emptyset
\\[4mm]
0 &
\text{otherwise}
\end{array}
\right. ,
\\[8mm]
f\left(\I{A}\right)(v) &= \left\{%
\arraycolsep=14pt   
\begin{array}{ll}
{\displaystyle \inf_{u\in \fibre{f}{v}} \DI{A} }  &
\text{if } \fibre{f}{v} \neq \emptyset
\\[4mm]
0 &
\text{otherwise}
\end{array}
\right. ,
\\[8mm]
f\left(\NM{A}\right)(v) &= \left\{%
\arraycolsep=
14pt   
\begin{array}{ll}
{\displaystyle \sup_{u\in \fibre{f}{v}} \DNM{A} }  &
\text{if } \fibre{f}{v} \neq \emptyset
\\[4mm]
1 &
\text{otherwise}
\end{array}
\right.
\end{array} 
$$
for every $v \in \V$.
\end{definition}

\begin{definition}{\rm\cite{latreche2020,salama2014}}
\label{def:neutrosophicinverseimage}
Let $f: \U \to \V$ be a mapping between two universe sets $\U$ and $\V$,
and $\nNS[\V]{B}$ be a \nameSVNS over $\V$.
The \df{neutrosophic inverse image} of $\ns{B}$ by $f$,
denoted by $\inv{\ns{f}}\left(\ns{B}\right)$, is the \nameSVNS over $\U$ defined by:
$$\inv{\ns{f}}\left(\ns{B}\right) = \NSbase{\inv{f}\left(\M{B}\right)}%
{\inv{f}\left(\I{B}\right)}%
{\inv{f}\left(\NM{B}\right)}$$
where the mappings $\inv{f}\left(\M{B}\right) : \U \to I$, $\inv{f}\left(\I{B}\right) : \U \to I$
and $\inv{f}\left(\NM{B}\right) : \U \to I$ are defined respectively by:
$$
\arraycolsep=1.5pt   
\begin{array}{ll}
\inv{f}\left(\M{B}\right)(u) &= \DM[f(u)]{B}  \, ,
\\[4mm]
\inv{f}\left(\I{B}\right)(u) &= \DI[f(u)]{B}  \, ,
\\[4mm]
\inv{f}\left(\NM{B}\right)(u) &= \DNM[f(u)]{B}
\end{array}
$$
for every $u \in \U$.
\end{definition}

\begin{remark}
Let us note that the notation used to denote neutrosophic images
and neutrosophic inverse images implicitly underlines the fact that they are not real images or
counterimages since the mapping $f$ is defined between the universe sets $\U$ and $\V$
while the definition refers to the sets $\SSVNS$ and $\SSVNS[\V]$
of all the \nameSVNS[s] of that respective universe sets.
More properly, this means that we consider a mapping $\ns{f}:\SSVNS \to \SSVNS[\V]$
induced by $f:\U \to \V$ over the corresponding sets of all \nameSVNS[s].
\end{remark}

\begin{example}
\label{ex:neutrosophicimagesandinverseimages}
Let $\U = \left\{ a,b,c \right\}$ and $\V = \left\{ \alpha, \beta, \gamma , \delta \right\}$
be two finite universe sets,
$f: \U \to \V$ be a mapping defined by $f(a)=f(c)=\beta$ and $f(b)=\alpha$.
Let us consider a \nameSVNS $\nNS{A}$ on $\SSVNS$ and a \nameSVNS $\nNS{B}$ on $\SSVNS[\V]$
respectively defined by the following tabular representations:
\vspace{-3mm}
\begin{center}
\begin{tabular}[t]{cc}
\begin{nstabular}{\U}{\ns{A}}{\M{A}}{\I{A}}{\NM{A}}
$a$ & 0.5 & 0.3 & 0.2
\\ \hline
$b$ & 0.6 & 0.2 & 0.3
\\ \hline
$c$ & 0.4 & 0.2 & 0.7
\end{nstabular}
\hspace{16mm} &
\begin{nstabular}{\V}{\ns{B}}{\M{B}}{\I{B}}{\NM{B}}
$\alpha$ & 0.1 & 0.7 & 0.9
\\ \hline
$\beta$ & 0.5 & 0.3 & 0.1
\\ \hline
$\gamma$ & 0.8 & 0.4 & 0.2
\\ \hline
$\delta$ & 0.4 & 0.6 & 0.8
\end{nstabular}
\\
\end{tabular}
\end{center}
Then the neutrosophic image $\ns{f}\left(\ns{A}\right) = \NSbase[\V]
{f\left(\M{A}\right)} {f\left(\I{A}\right)} {f\left(\NM{A}\right)}$
of $\ns{A}$ by $f$
and the neutrosophic inverse image
$\inv{\ns{f}}\left(\ns{B}\right) = \NSbase{\inv{f}\left(\M{B}\right)}%
{\inv{f}\left(\I{B}\right)} {\inv{f}\left(\NM{B}\right)}$
of $\ns{B}$ by $f$ are given by:
\vspace{-3mm}
\begin{center}
\begin{nstabular}{\V}{\ns{f}\left(\ns{A}\right)}%
{f\left(\M{A}\right)}%
{f\left(\I{A}\right)}%
{f\left(\NM{A}\right)}
$\alpha$ & 0.6 & 0.2 & 0.3
\\ \hline
$\beta$ & 0.4 & 0.2 & 0.7
\\ \hline
$\gamma$ & 0 & 0 & 1
\\ \hline
$\delta$ & 0 & 0 & 1
\end{nstabular}
\end{center}
and:
\vspace{-3mm}
\begin{center}
\begin{nstabular}{\U}{\quad\inv{\ns{f}}\left(\ns{B}\right) }%
{\inv{f}\left(\M{B}\right)}%
{\inv{f}\left(\I{B}\right)}%
{\inv{f}\left(\NM{B}\right)}
$a$ & 0.5 & 0.3 & 0.1
\\ \hline
$b$ & 0.1 & 0.7 & 0.9
\\ \hline
$c$ & 0.5 & 0.3 & 0.1
\end{nstabular}
\end{center}
respectively.
\end{example}

\begin{proposition}{\rm\cite{salama2014}}
\label{pro:propertiesofneutrosophicimagesandinverseimages}
Let $f: \U \to \V$ be a mapping between two universe sets $\U$ and $\V$,
$\nNS{A}$ be a \nameSVNS over $\U$
and $\nNS[\V]{B}$ be a \nameSVNS over $\V$, then the following hold:
\begin{enumi}
\item $\ns{f}\left( \NSemptyset[\U]\right) \NSeq \NSemptyset[\V]$
\item $\inv{\ns{f}}\left( \NSemptyset[\V]\right) \NSeq \NSemptyset[\U]$
\item $\inv{\ns{f}}\left( \NSabsoluteset[\V]\right) \NSeq \NSabsoluteset$
\item $\ns{A} \NSsubseteq \inv{\ns{f}}\left( \ns{f}\left(\ns{A}\right) \right)$
and the identity holds if $\ns{f}$ is injective
\item $\ns{f}\left( \inv{\ns{f}}\left(\ns{B}\right) \right) \NSsubseteq \ns{B}$
and the identity holds if $\ns{f}$ is surjective
\item $\inv{\ns{f}}\left( \ns{B}^\NScompl \, \right)
\NSeq \inv{\ns{f}}\left( \ns{B} \right)\!^\NScompl$
\end{enumi}
\end{proposition}

\begin{proposition}{\rm\cite{salama2014}}
\label{pro:monotonicfneutrosophicimagesandinverseimages}
Let $f: \U \to \V$ be a mapping between two universe sets $\U$ and $\V$,
$\ns{A}_i=\NS{A_i}$ (with $i=1,2$) be \nameSVNS[s] over $\U$
and $\ns{B}_i=\NS[\V]{B_i}$ (with $i=1,2$) be \nameSVNS[s] over $\V$, then the following hold:
\begin{enumi}
\item if $\ns{A}_1 \NSsubseteq \ns{A}_2$ then
$\ns{f}\left(\ns{A}_1\right) \NSsubseteq \ns{f}\left(\ns{A}_2\right)$
\item if $\ns{B}_1 \NSsubseteq \ns{B}_2$ then
$\inv{\ns{f}}\left(\ns{B}_1\right) \NSsubseteq \inv{\ns{f}}\left(\ns{B}_2\right)$
\end{enumi}
\end{proposition}

\begin{proposition}{\rm\cite{salama2014}}
\label{pro:propertiesofneutrosophicimagesandinverseimagesofgeneralizedunionandintersections}
Let $f: \U \to \V$ be a mapping between two universe sets $\U$ and $\V$,
$\left\{ \ns{A}_i \right\}_{i\in I}$ be a family of \nameSVNS[s]
$\ns{A}_i=\NS{A_i}$ over $\U$
and $\left\{ \ns{B}_i \right\}_{i\in I}$ be a family of \nameSVNS[s]
$\ns{B}_i=\NS[\V]{B_i}$ over $\V$ ,
then the following hold:
\begin{enumi}
\item $\ns{f} \left( \NSCup_{i \in I} \ns{A}_i \right) \NSeq \NSCup_{i \in I} \ns{f} \left( \ns{A}_i \right)$
\item $\ns{f} \left( \NSCap_{i \in I} \ns{A}_i \right) \NSsubseteq
       \NSCap_{i \in I} \ns{f} \left( \ns{A}_i \right)$
and the identity holds if $\ns{f}$ is injective
\item $\inv{\ns{f}} \left( \NSCup_{i \in I} \ns{A}_i \right) \NSeq \NSCup_{i \in I} \inv{\ns{f}} \left( \ns{A}_i \right)$
\item $\inv{\ns{f}} \left( \NSCap_{i \in I} \ns{A}_i \right) \NSeq \NSCap_{i \in I} \inv{\ns{f}} \left( \ns{A}_i \right)$
\end{enumi}
\end{proposition}

\section{Single Valued Neutrosophic Filters}

In \cite{salama2013}, A.A. Salama and H. Alagamy introduced the notion of filter on neutrosophic set.
That study, unfortunately, is rather incomplete and certainly not exhaustive because it does not cover
the neutrosophic equivalent of many fundamental notions and properties such as
those concerning the lower or upper bound of two filter bases,
the proof of the ultrafilter's existence, etc.
For this reason, together with the fact that -- as already mentioned -- the class
of the single valued neutrosophic sets lends itself with greater ductility to resolution
of real life problems and applications,
we give here a comprehensive presentation about theory of filters on \nameSVNS[s]
which includes new notions and properties that are not present in that article.

\begin{definition}
\label{def:singlevaluedneutrosophicfiltersubbase}
Let $\cA \subseteq \SSVNS$ be a nonempty family of \nameSVNS[s] over the universe set $\U$,
we say that $\cF$ is a \df{single valued neutrosophic filter subbase} (\df{SVN-filter subbase} for short,
or simply a filter subbase) on $\SSVNS$
if it has the finite intersection property,
i.e. if for every $\ns{A_1}, \ldots \ns{A_n} \in \cA$
(with $\ns{A}_i=\NS{A_i}$ for every $i=1,\ldots n$, and $n \in \NN^*$),
it results $\NSCap_{i=1}^n \ns{A}_i \NSneq \NSemptyset$.
\end{definition}

\begin{definition}
\label{def:singlevaluedneutrosophicfilterbase}
A nonempty family $\cF \subseteq \SSVNS$ of \nameSVNS[s] over the universe set $\U$
is \df{single valued neutrosophic filter base} (\df{SVN-filter base} for short,
or simply a filter base) on $\SSVNS$
if the following two conditions hold:
\begin{enumr}
\item $\NSemptyset \notin \cF$
\item for every $\ns{F}, \ns{G} \in \cF$ there exists some $\ns{H} \in \cF$
such that $\ns{H} \NSsubseteq \ns{F} \NScap \ns{G}$.
\end{enumr}
\end{definition}

\begin{definition}
\label{def:singlevaluedneutrosophicfilter}
A nonempty family $\cF \subseteq \SSVNS$ of \nameSVNS[s] over the universe set $\U$
is \df{single valued neutrosophic filter} (\df{SVN-filter} for short or simply a filter) on $\SSVNS$
if:
\begin{enumr}
\item $\cF$ is a SVN-filter base, and
\item for every $\nNS{F} \in \cF$ and every $\nNS{A} \in \SSVNS$ such that $\ns{F} \NSsubseteq \ns{A}$
it follows that $\ns{A} \in \cF$.
\end{enumr}
\end{definition}

An equivalent definition of SVN-filter is given by the following proposition.

\begin{proposition}
\label{pro:characterizationofneutrosophicfilter}
A nonempty family $\cF \subseteq \SSVNS$ of \nameSVNS[s] over the universe set $\U$
is a SVN-filter on $\SSVNS$ if and only if the following three conditions hold:
\begin{enumr}
\item $\NSemptyset \notin \cF$
\item for every $\ns{F}, \ns{G} \in \cF$, it follows that $\ns{F} \NScap \ns{G} \in \cF$
\item for every $\ns{F} \in \cF$ and every $\ns{A} \in \SSVNS$ such that $\ns{F} \NSsubseteq \ns{A}$
we have that $\ns{A} \in \cF$.
\end{enumr}
\end{proposition}
\begin{proof}
Let $\cF$ be a SVN-filter over $\U$.
Evidently, conditions (i) and (iii) are satisfied.
Let $\ns{F}, \ns{G} \in \cF$. By condition (ii) of the definition of SVN-filter base,
there exists some $\ns{H} \in \cF$ such that $\ns{H} \NSsubseteq \ns{F} \NScap \ns{G}$
and so, by the peculiar condition of SVN-filter, it also follows that $\ns{F} \NScap \ns{G} \in \cF$.

Conversely, if $\cF$ is a nonempty family of \nameSVNS[s] over the universe set $\U$
satisfying conditions (i), (ii) and (iii), it is enough to note that (ii)
implies condition (ii) of Definition \ref{def:singlevaluedneutrosophicfilterbase} and so that
$\cF$ is SVN-filter over $\U$.
\end{proof}

It is a simple routine to show that the condition (ii) of the
Proposition \ref{pro:characterizationofneutrosophicfilter}
can be generalized to any finite neutrosophic intersection
as it is pointed out in the following corollary.

\begin{corollary}
\label{cor:characterizationofneutrosophicfilterwithfiniteintersection}
A nonempty family $\cF \subseteq \SSVNS$ of \nameSVNS[s] over the universe set $\U$
is a SVN-filter on $\SSVNS$ if and only if the following three conditions hold:
\begin{enumr}
\item $\NSemptyset \notin \cF$
\item for every $\ns{F}_1, \ldots \ns{F}_n \in \cF$, it follows that $\NSCap_{i \in I} \ns{F}_i \in \cF$
\item for every $\ns{F} \in \cF$ and every $\ns{A} \in \SSVNS$ such that $\ns{F} \NSsubseteq \ns{A}$
we have that $\ns{A} \in \cF$.
\end{enumr}
\end{corollary}

\begin{remark}
\label{rem:relationshipbetweenfilterandfilterbase}
Evidently every SVN-filter is a SVN-filter base and every SVN-filter base is a SVN-filter subbase.
It is also trivial to note that every SVN-filter on $\SSVNS$ contains $\U$.
\end{remark}


\begin{example}
\label{ex:singlevaluedneutrosophicfilterbasenotfilter}
Let $\U = \left\{ a,b,c \right\}$ be a finite universe set and
let $\cF = \left\{ \ns{F}, \ns{G}, \ns{H}, \NSabsoluteset \right\} \subseteq \SSVNS$
be a set of the \nameSVNS[s] $\nNS{F}, \nNS{G}, \nNS{H}$ and
$\NSabsoluteset = \NSbase{\und{1}}{\und{1}}{\und{0}}$
respectively defined by the following tabular representations:
\vspace{-3mm}
\begin{center}
\begin{tabular}{cc}
\begin{nstabular}{\U}{\ns{F}}{\M{F}}{\I{F}}{\NM{F}}
$a$ & 0.4 & 0.3 & 0.2
\\ \hline
$b$ & 0.8 & 0.2 & 0.1
\\ \hline
$c$ & 0.6 & 0.5 & 0.4
\end{nstabular}
\hspace{12mm} &
\begin{nstabular}{\U}{\ns{G}}{\M{G}}{\I{G}}{\NM{G}}
$a$ & 0.7 & 0.1 & 0.3
\\ \hline
$b$ & 0.9 & 0.2 & 0.2
\\ \hline
$c$ & 0.2 & 0.6 & 0.5
\end{nstabular}
\\[12mm]
\begin{nstabular}{\U}{\ns{H}}{\M{H}}{\I{H}}{\NM{H}}
$a$ & 0 & 0.4 & 0.8
\\ \hline
$b$ & 0.5 & 0.3 & 0.6
\\ \hline
$c$ & 0.1 & 0.8 & 0.5
\end{nstabular}
\hspace{12mm} &
\begin{nstabular}{\U}{\NSabsoluteset}{\M{\U}}{\I{\U}}{\NM{\U}}
$a$ & 1 & 1 & 0
\\ \hline
$b$ & 1 & 1 & 0
\\ \hline
$c$ & 1 & 1 & 0
\end{nstabular}
\end{tabular}
\end{center}
It is easy to check that $\cF$ is a SVN-filter base on $\SSVNS$.
However, by using Proposition \ref{pro:characterizationofneutrosophicfilter},
we have that $\cF$ is not a SVN-filter since, for example, the neutrosophic intersection
$\ns{W} \NSeq \ns{F} \NScap \ns{G}$ of $\ns{F}, \ns{G} \in \cF$
computated as shown in the following tabular representation:
\vspace{-3mm}
\begin{center}
\begin{nstabular}{\U}{\ns{W}}{\M{W}}{\I{W}}{\NM{W}}
$a$ & 0.4 & 0.3 & 0.3
\\ \hline
$b$ & 0.8 & 0.2 & 0.2
\\ \hline
$c$ & 0.2 & 0.6 & 0.5
\end{nstabular}
\end{center}
is a \nameSVNS over $\U$ which does not belong to the family $\cF$.
\end{example}


\begin{notation}
\label{not:setofsinglevaluedneutrosophicfilters}
The set of all the single valued neutrosophic filters over the universe set $\U$
will be denoted by $\SVNF$.
\end{notation}

\begin{definition}
\label{def:comparisonofneutrosophicfilterbase}
Let $\cF$ and $\cG$ be two SVN-filter bases on $\SSVNS$,
we say that $\cG$ is \df{finer} than $\cF$ if $\cF \subseteq \cG$.
We also say that $\cF$ is \df{coarser} than $\cG$.
\end{definition}

Let us note that the set $\SVNF$ equipped with the finess 
relation $\subseteq$ 
forms a poset although it is not a loset.


\begin{proposition}
\label{pro:neutrosophicfilterbasefromsubbase}
Let $\cS$ be a SVN-filter subbase on $\SSVNS$ and let
$$\cS^* = \left\{ \NSCap_{i=1}^n \ns{A_i} : \, \nNS{A_i} \in \cS, \, n \in \NN^* \right\}$$
be the set of all finite neutrosophic intersections of $\cS$.
Then $\cS^*$ is a SVN-filter base on $\SSVNS$
containing $\cS$, i.e. $\cS \subseteq \cS^*$.
\end{proposition}
\begin{proof}
Let $\cS$ be a SVN-filter subbase over $\U$.
Since, by Definition \ref{def:singlevaluedneutrosophicfiltersubbase}, $\cS$ has the finite intersection property,
it is evident that $\cS^*$ satisfies the condition (i) of Definition \ref{def:singlevaluedneutrosophicfilterbase}.
Furthermore, for every $\ns{A}, \ns{B} \in \cS^*$,
there exist some $m, n \in \NN^*$, $\nNS{A_i} \in \cS$ (with $i=1,\ldots m$)
and $\nNS{B_j} \in \cS$ (with $j=1,\ldots n$) such that
$\ns{A} = \NSCap_{i=1}^m \ns{A_i}$ and $\ns{B} = \NSCap_{j=1}^n \ns{B_i}$
and so, by definition of $\cS^*$, it is clear that $\ns{A} \NScap \ns{B} \in \cS^*$.
Thus, $\cS^*$ also satisfies condition (ii) of Definition \ref{def:singlevaluedneutrosophicfilterbase}
and it is a SVN-filter base.
Finally, for every $\ns{A} \in \cS$, since by Proposition \ref{pro:propertiesunionandintersection}\,(4)
it results $\ns{A} \NScap \ns{A} = \ns{A}$, we have that $\ns{A} \in \cS^*$  and hence that
$\cS \subseteq \cS^*$.
\end{proof}

\begin{definition}
\label{def:neutrosophicfiltersubbase}
Let $\cS$ be a SVN-filter subbase on $\SSVNS$ , the SVN-filter base $\cS^*$
defined in the proposition above is called the \df{neutrosophic filter base generated}
by its \df{neutrosophic filter subbase} $\cS$.
\end{definition}

\begin{proposition}
\label{pro:neutrosophicfiltercompletion}
Let $\cF$ be a SVN-filter base on $\SSVNS$ and let
$$\completion{\cF} = \left\{ \nNS{A} \in \SSVNS : \,\, \exists \nNS{F} \in \cF , \,\,
\ns{F} \NSsubseteq \ns{A} \right\}$$
be the set of all neutrosophic supersets of members of $\cF$,
then $\completion{\cF}$ is a SVN-filter on $\SSVNS$
containing $\cF$, i.e. $\cF \subseteq \completion{\cF}$.
\end{proposition}
\begin{proof}
Let $\cF$ be a SVN-filter base over $\U$.
Evidently, every member of $\completion{\cF}$ is nonempty
and for every $\ns{A}, \ns{B} \in \completion{\cF}$, there exist some $\ns{F}, \ns{G} \in \cF$
such that $\ns{F} \NSsubseteq \ns{A}$ and $\ns{G} \NSsubseteq \ns{B}$.
So, by Proposition \ref{pro:monotonic_neutrosophic_operators}\,(2), it follows that
$\ns{F} \NScap \ns{G} \NSsubseteq \ns{A} \NScap \ns{B}$
and, by condition (ii) of Definition \ref{def:singlevaluedneutrosophicfilterbase},
we have that there exists some $\ns{H} \in \cF$
such that $\ns{H} \NSsubseteq \ns{F} \NScap \ns{G}$ which implies
that $\ns{H} \NSsubseteq \ns{A} \NScap \ns{B}$
and hence that $\ns{A} \NScap \ns{B} \in \completion{\cF}$.
Furthermore, for every $\ns{A} \in \completion{\cF}$ and $\ns{B} \in \SSVNS$
such that $\ns{A} \NSsubseteq \ns{B}$, we have that there exists some $\ns{F} \in \cF$
such that $\ns{F} \NSsubseteq \ns{A}$ and, consequently, $\ns{F} \NSsubseteq \ns{B}$
which means that $\ns{B} \in \completion{\cF}$.
Thus, $\completion{\cF}$ satisfies all the conditions of Proposition \ref{pro:characterizationofneutrosophicfilter}
and so it is a SVN-filter on  $\SSVNS$.
Furthemore, we have that $\cF \subseteq \completion{\cF}$ since it is clear that
for every $\ns{A} \in \cF$, it results $\ns{A} \NSsubseteq \ns{A}$
and hence $\ns{A} \in \completion{\cF}$.
\end{proof}

\begin{definition}
\label{def:neutrosophicfiltercompletion}
Let $\cF$ be a SVN-filter base on $\SSVNS$, the SVN-filter $\completion{\cF}$
defined in the proposition above is called the \df{neutrosophic filter completion} of $\cF$.
Additionally, we say that $\cF$ is a \df{neutrosophic filter base} for the SVN-filter $\completion{\cF}$.
\end{definition}

\begin{proposition}
\label{pro:monotoniccompletion}
If $\cF$ and $\cG$ are two SVN-filter bases on $\SSVNS$ such that $\cF \subseteq \cG$
then $\completion{\cF} \subseteq \completion{\cG}$.
\end{proposition}
\begin{proof}
In fact, for every $\ns{A} \in \completion{\cF}$ we have that
there exists some $\ns{F} \in \cF$ such that $\ns{F}\NSsubseteq \ns{A}$
and since $\cF\subseteq \cG$ it also follows that $\ns{F} \in \cG$
and hence that $\ns{A} \in \completion{\cG}$.
\end{proof}

\begin{definition}
\label{def:equivalentneutrosophicfilterbase}
Let $\cF$ and $\cG$ be two SVN-filter bases on $\SSVNS$.
We say that $\cF$ and $\cG$ are \df{equivalent} if they are both neutrosophic filter base
for the same SVN-filter, that is if $\completion{\cF} = \completion{\cG}$.
\end{definition}

\begin{remark}
\label{rem:finiteinteserctionandcompletion}
It is a simple matter to verify that:
\begin{enumi}
\item if $\cS$ is a SVN-filter base on $\SSVNS$ then $\cS^* = \cS$,
\item if $\cF$ is a SVN-filter on $\SSVNS$ then $\cF^* = \cF$ and $\completion{\cF} = \cF$.
\end{enumi}
\end{remark}

\begin{proposition}
\label{pro:characterizationneutrosophicfiltergenerated}
If $\cS$ is a SVN-filter subbase on $\SSVNS$ then
$\completion{\cS^*}$ is the coarsest SVN-filter on $\SSVNS$
containing $\cS$, i.e. such that:
\begin{enumr}
\item $\cS \subseteq \completion{\cS^*}$, and
\item for every SVN-filter $\cH$ on $\SSVNS$ such that $\cS \subseteq \cH$
it follows that $\completion{\cS^*} \subseteq \cH$.
\end{enumr}
\end{proposition}
\begin{proof}
Let $\cS$ be a SVN-filter subbase over $\U$.
Condition (i) is trivially verified since by Propositions \ref{pro:neutrosophicfilterbasefromsubbase}
and \ref{pro:neutrosophicfiltercompletion}, we immediately have that $\cS \subseteq \cS^* \subseteq \completion{\cS^*}$.
Now, suppose that $\cH$ is a SVN-filter on $\SSVNS$ such that $\cS \subseteq \cH$
and let $\ns{A} \in \completion{\cS^*}$.
Then, for some $n \in \NN^*$, there exist $\ns{B}_1, \ldots \ns{B}_n \in \cS$ such that
$\NSCap_{i=1}^n \ns{B}_i \NSsubseteq \ns{A}$.
Since $\cS \subseteq \cH$, it follows that every $\ns{B}_i \in \cH$ (for $i=1,\ldots n$) and,
by Corollary \ref{cor:characterizationofneutrosophicfilterwithfiniteintersection},
we obtain that $\NSCap_{i=1}^n \ns{B}_i \in \cH$ and
therefore that $\ns{A} \in \cH$
which proves condition (ii) and concludes our proof.
\end{proof}

\begin{definition}
\label{def:neutrosophicfiltergenerated}
Let $\cS$ be a SVN-filter subbase on $\SSVNS$, the SVN-filter $\completion{\cS^*}$
defined in the proposition above is called the \df{neutrosophic filter generated}
by its neutrosophic filter subbase $\cS$.
\end{definition}

In particular, if $\nNS{A}$ is a nonempty \nameSVNS over the universe set $\U$,
the SVN-filter $\completion{\cS^*}$ generated by the singleton $\cS = \left\{ \ns{A} \right\}$,
being the coarser (smallest) \nameSVNS containing $\cS$,
coincides with the family of all single valued neutrosophic superset of $\ns{A}$,
is denoted simply with $\completion{\ns{A}}$
and is called the \df{SVN-principal filter} generated by $\ns{A}$.

\begin{proposition}
\label{pro:neutrosophicprincipalfilter}
If $\cF$ is a finite SVN-filter base on $\SSVNS$,
then the neutrosophic filter completion $\completion{\cF}$
is a SVN-principal filter over $\U$.
\end{proposition}
\begin{proof}
Let $\cF = \left\{ \ns{F}_1, \ldots \ns{F}_n \right\}$ (with $\ns{F}_i = \NS{F_i}$, $i=1,\ldots n$)
be a finite SVN-filter base
and let $\ns{G} \NSeq \NSCap_{i=1}^n \ns{F}_i$.
We will show that $\cG = \left\{ \ns{G} \right\}$ is an equivalent SVN-filter base for
the SVN-filter $\completion{\cF}$.
In fact, since $\cF$ is a SVN-filter base, by Remark \ref{rem:finiteinteserctionandcompletion}\,(1),
we have that $\ns{G} \in \cF^* = \cF$.
Thus $\cG \subseteq \cF$ and by Proposition \ref{pro:monotoniccompletion} it follows that
$\completion{\cG} \subseteq \completion{\cF}$.
On the other hand, for every $\ns{A} \in \completion{\cF}$, we have that
there exists some $j = 1,\ldots n$ such that $\ns{F}_j \NSsubseteq \ns{A}$
and since by Proposition \ref{pro:unionandintersectiongeneralized} we know that
$\ns{G} = \NSCap_{i=1}^n \ns{F}_i \NSsubseteq \ns{F}_j$,
it follows that $\ns{G} \NSsubseteq \ns{A}$ and so that
$\ns{A} \in \completion{\cG}$.
This proves that $\completion{\cF} \subseteq \completion{\cG}$
and consequently that $\completion{\cF} = \completion{\cG} = \completion{\ns{G}}$,
i.e. that $\completion{\cF}$ is a SVN-principal filter generated by $\ns{G}$.
\end{proof}


\begin{proposition}
\label{pro:infneutrosophicfilterbase}
Let $\cF$ and $\cG$ be two SVN-filter bases on $\SSVNS$ and let
$$\cF \wedge \cG = \left\{ \ns{F} \NScup \ns{G} : \,\, \nNS{F} \in \cF, \, \nNS{G} \in \cG \right\}$$
be the set of all neutrosophic unions of the members of $\cF$ and $\cG$, then
$\cF \wedge \cG$ is a SVN-filter base on $\SSVNS$.
\\
Additionally, if $\cF$ and $\cG$ are SVN-filter over $\U$ then $\cF \wedge \cG$ is a SVN-filter on $\SSVNS$
which is coarser than both $\cF$ and $\cG$, i.e.
$\cF \wedge \cG \subseteq \cF$ and $\cF \wedge \cG \subseteq \cG$.
\end{proposition}
\begin{proof}
If $\cF$ and $\cG$ are two SVN-filter bases on $\SSVNS$,
for every $\ns{F} \in \cF$ and $\ns{G} \in \cG$, it is evident that
$\ns{F} \NScup \ns{G} \NSneq \NSemptyset$ and so that
$\cF \wedge \cG$ verifies the condition (i) of Definition \ref{def:singlevaluedneutrosophicfilterbase}.
Moreover, for every $\ns{A}_1, \ns{A}_2 \in \cF \wedge \cG$, we have that there exist
some $\ns{F}_1, \ns{F}_2 \in \cF$ and $\ns{G}_1, \ns{G}_2 \in \cG$
such that $\ns{A}_1 \NSeq \ns{F}_1 \NScup \ns{G}_1$ and $\ns{A}_2 \NSeq \ns{F}_2 \NScup \ns{G}_2$.
Since $\cF$ and $\cG$ are SVN-filter bases, there exist
$\ns{F}_3 \in \cF$ and $\ns{G}_3 \in \cG$ such that
$\ns{F}_3 \NSsubseteq \ns{F}_1 \NScap \ns{F}_2$
and $\ns{G}_3 \NSsubseteq \ns{G}_1 \NScap \ns{G}_2$.
So, $\ns{F}_3 \NScup \ns{G}_3 \in \cF \wedge \cG$ and,
by using Proposition \ref{pro:generalizeddistributive}, it results
$\ns{F}_3 \NScup \ns{G}_3 \NSsubseteq
\left( \ns{F}_1 \NScap \ns{F}_2 \right) \NScup \left( \ns{G}_1 \NScap \ns{G}_2 \right)
\NSeq  \left( \ns{F}_1 \NScup \left( \ns{G}_1 \NScap \ns{G}_2 \right) \right)
\NScap \left( \ns{F}_2 \NScup \left( \ns{G}_1 \NScap \ns{G}_2 \right) \right)
\NSsubseteq \left( \ns{F}_1 \NScup \ns{G}_1 \right)
\NScap \left( \ns{F}_2 \NScup \ns{G}_2 \right)
\NSeq \ns{A}_1 \NScap \ns{A}_2$
and this means that $\cF \wedge \cG$ also verifies condition (ii)
of Definition \ref{def:singlevaluedneutrosophicfilterbase}
and hence that it is a SVN-filter base on $\SSVNS$.
\\
Now, suppose that $\cF$ and $\cG$ are are SVN-filters
and let $\ns{F} \in \cF$,  $\ns{G} \in \cG$ and $\ns{A} \in \SSVNS$
such that $\ns{F} \NScup \ns{G} \NSsubseteq \ns{A}$.
Since $\cF$ is a SVN-filter and $\ns{F} \NSsubseteq \ns{F} \NScup \ns{G} \NSsubseteq \ns{A}$,
we have that $\ns{A} \in \cF$.
Analogously, by the fact that $\cG$ is a SVN-filter and $\ns{G} \NSsubseteq \ns{F} \NScup \ns{G} \NSsubseteq \ns{A}$, we have that $\ns{A} \in \cG$.
Thus $\ns{A} = \ns{A} \NScup \ns{A} \in \cF \wedge \cG$
and this proves that $\cF \wedge \cG$ is a SVN-filter over $\U$.
\\
In such a situation, for every $\ns{F} \NScup \ns{G} \in \cF \wedge \cG$,
with $\ns{F} \in \cF$ and $\ns{G} \in \cG$,
being $\ns{F} \NSsubseteq \ns{F} \NScup \ns{G}$,
we have that $\ns{F} \NScup \ns{G} \in \cF$ and so that $\cF \wedge \cG \subseteq \cF$.
In a similar way, one can also proves that $\cF \wedge \cG \subseteq \cG$.
\end{proof}

\begin{proposition}
\label{pro:supneutrosophicfilterbase}
Let $\cF$ and $\cG$ be two SVN-filter bases on $\SSVNS$ such that $\cF$ nuetrosophically meets $\cG$ and let
$$\cF \vee \cG = \left\{ \ns{F} \NScap \ns{G} : \,\, \nNS{F} \in \cF, \, \nNS{G} \in \cG \right\}$$
be the set of all neutrosophic intersections of the members of $\cF$ and $\cG$, then
$\cF \vee \cG$ is a SVN-filter base on $\SSVNS$.
\\
Additionally, if $\cF$ and $\cG$ are SVN-filters over $\U$ then $\cF \vee \cG$ is a SVN-filter on $\SSVNS$
which is finer than both $\cF$ and $\cG$, i.e.
$\cF \subseteq \cF \vee \cG$ and $\cG \subseteq \cF \vee \cG$.
\end{proposition}
\begin{proof}
Since $\cF$ neutrosophically meets $\cG$, it is clear that $\NSemptyset \notin \cF \vee \cG$,
i.e. that $\cF \vee \cG$ verifies the condition (i) of Definition \ref{def:singlevaluedneutrosophicfilterbase}.
Moreover, for every $\ns{A}_1, \ns{A}_2 \in \cF \vee \cG$, we have that there exist
some $\ns{F}_1, \ns{F}_2 \in \cF$ and $\ns{G}_1, \ns{G}_2 \in \cG$
such that $\ns{A}_1 \NSeq \ns{F}_1 \NScap \ns{G}_1$ and $\ns{A}_2 \NSeq \ns{F}_2 \NScap \ns{G}_2$.
Since $\cF$ and $\cG$ are SVN-filter bases, there exist
$\ns{F}_3 \in \cF$ and $\ns{G}_3 \in \cG$ such that
$\ns{F}_3 \NSsubseteq \ns{F}_1 \NScap \ns{F}_2$
and $\ns{G}_3 \NSsubseteq \ns{G}_1 \NScap \ns{G}_2$.
So, $\ns{F}_3 \NScap \ns{G}_3 \in \cF \vee \cG$ and it results
$\ns{F}_3 \NScap \ns{G}_3 \NSsubseteq
\left( \ns{F}_1 \NScap \ns{F}_2 \right) \NScap \left( \ns{G}_1 \NScap \ns{G}_2 \right)
\NSeq \left( \ns{F}_1 \NScap \ns{G}_1 \right) \NScap \left( \ns{F}_2 \NScap \ns{G}_2 \right)
\NSeq \ns{A}_1 \NScap \ns{A}_2$ and this means that $\cF \vee \cG$ also verifies
the condition (ii) of Definition \ref{def:singlevaluedneutrosophicfilterbase} is  verified
and hence that it is a SVN-filter base on $\SSVNS$.
\\
Now, suppose that $\cF$ and $\cG$ are SVN-filters
and let $\ns{F} \in \cF$,  $\ns{G} \in \cG$ and $\ns{A} \in \SSVNS$
such that $\ns{F} \NScap \ns{G} \NSsubseteq \ns{A}$.
Since $\ns{F} \NSeq \ns{F} \NScup \left( \ns{F} \NScap \ns{G} \right)
\NSsubseteq \ns{F} \NScup \ns{A}$
and $\cF$ is a SVN-filter, we have that $\ns{F} \NScup \ns{A} \in \cF$.
In a similar way, since  $\ns{G} \NSsubseteq \ns{G} \NScup \ns{A}$
and $\cG$ is a SVN-filter, it follows that $\ns{G} \NScup \ns{A} \in \cG$
and hence that
$\left( \ns{F} \NScup \ns{A} \right) \NScap \left( \ns{G} \NScup \ns{A} \right) \in \cF \vee \cG$.
By Proposition \ref{pro:generalizeddistributive}\,(2)
and Proposition \ref{pro:neutrosophicsubset_and_neutrosophicoperators}\,(2)
we have that
$\left( \ns{F} \NScup \ns{A} \right) \NScap \left( \ns{G} \NScup \ns{A} \right)
\NSeq  \left( \ns{F} \NScap \ns{G} \right) \NScup  \ns{A} = \ns{A}$
and so that $\ns{A} \in \cF \vee \cG$ which proves that $\cF \vee \cG$ is a SVN-filter over $\U$.
\\
In such a situation, for every $\ns{F} \in \cF$ and for any fixed $\ns{G} \in \cG$, we have that
$\ns{F} \NScap \ns{G} \NSsubseteq \ns{F}$ with $\ns{F} \NScap \ns{G} \in \cF \vee \cG$
and so that also $\ns{F} \in \cF \vee \cG$ which proves that $\cF \subseteq \cF \vee \cG$.
In a similar way, one can also proves that $\cG \subseteq \cF \vee \cG$.
\end{proof}


\begin{proposition}
\label{pro:intersectionofaneutrosophicsetwithaneutrosophicfilterbase}
Let $\cF$ be a SVN-filter base on $\SSVNS$
and $\nNS{A}$ be a \nameSVNS over $\U$ which neutrosophically meets $\cF$,
then the set
$\cF \vee \ns{A} = \left\{ \ns{F} \NScap \ns{A} : \,\, \ns{F} \in \cF \right\} $
of all neutrosophic intersections of $\ns{A}$ with the members of $\cF$
is a SVN-filter base over $\U$.
Additionally, if $\cF$ is a SVN-filter then $\cF \vee \ns{A}$ is a SVN-filter on $\SSVNS$
which is finer than $\cF$, i.e. $\cF \subseteq\cF \vee \ns{A}$.
\end{proposition}
\begin{proof}
Since by hypothesis $\cF$ neutrosophically meets $\nNS{A}$ it is evident that
$\NSemptyset \notin \cF \vee \ns{A}$.
Moreover, for every $\ns{G}_1, \ns{G}_2 \in \cF \vee \ns{A}$
there exist $\ns{F}_1, \ns{F}_2 \in \cF$ such that
$\ns{G}_1 = \ns{F}_1 \NScap \ns{A}$ and $\ns{G}_2 = \ns{F}_2 \NScap \ns{A}$.
Since $\cF$ is a SVN-filter base, there exists some $\ns{F}_3 \in \cF$ such that
such that $\ns{F}_3 \NSsubseteq \ns{F}_1 \NScap \ns{F}_2$.
So, let $\ns{G}_3 = \ns{F}_3 \NScap \ns{A}$, we note that $\ns{G}_3 \in \cF \vee \ns{A}$
and it results $\ns{G}_3 \NSsubseteq \ns{G}_1 \NScap \ns{G}_2$ which
proves that $\cF \vee \ns{A}$ is a SVN-filter base on $\SSVNS$.
\\
Now, suppose that $\cF$ is a SVN-filter over $\U$
and let $\ns{F} \in \cF$ and $\ns{B} \in \SSVNS$
such that $\ns{F} \NScap \ns{A} \NSsubseteq \ns{B}$,
with $\ns{B} \NSsubseteq \ns{A}$.
Since $\ns{F} \NSeq \ns{F} \NScup \left( \ns{F} \NScap \ns{A} \right)
\NSsubseteq \ns{F} \NScup \ns{B}$
and $\cF$ is a SVN-filter, we have that $\ns{F} \NScup \ns{B} \in \cF$
and hence that $\left( \ns{F} \NScup \ns{B} \right) \NScap \ns{A} \in \cF \vee \ns{A}$.
Moreover, by Proposition \ref{pro:generalizeddistributive}\,(2)
and Proposition \ref{pro:neutrosophicsubset_and_neutrosophicoperators}\,(2),
$\ns{B} \NSsubseteq \ns{A}$ and $\ns{F} \NScap \ns{A} \NSsubseteq \ns{B}$
imply that
$\left( \ns{F} \NScup \ns{B} \right) \NScap \ns{A} \NSeq
\left( \ns{F} \NScap \ns{A} \right) \NScup \left( \ns{B} \NScap \ns{A} \right)  \NSeq
\left( \ns{F} \NScap \ns{A} \right) \NScup \ns{B} \NSeq \ns{B}$
and hence that $\ns{B} \in \cF \vee \ns{A}$
which proves that $\cF \vee \ns{A}$ is a SVN-filter over $\U$.
\\
In such a situation, for every $\ns{F} \in \cF$, we have that
$\ns{F} \NScap \ns{A} \in \cF \vee \ns{A}$
and since $\ns{F} \NScap \ns{A} \NSsubseteq \ns{F}$ and $\cF \vee \ns{A}$ is a SVN-filter,
it follows that also $\ns{F} \in \cF \vee \ns{A}$ and hence that
$\cF \subseteq\cF \vee \ns{A}$.
\end{proof}


\begin{proposition}
\label{pro:imageofneutrosophicfilterbase}
Let $f: \U \to \V$ be a mapping between two universe sets $\U$ and $\V$
and let $\cF$ be a SVN-filter base on $\SSVNS$.
Then, the family
$\ns{f}\left(\cF\right) = \left\{ \ns{f}\left(\ns{F}\right) : \,\, \ns{F} \in \cF \right\}$
of all neutrosophic images on $\SSVNS$ by the mapping $f$
is a neutrosophic filter base on $\SSVNS[\V]$.
\end{proposition}
\begin{proof}
Let us consider a mapping $f: \U \to \V$
and a SVN-filter base $\cF$ over $\U$.
Evidently, for every $\ns{A} \in \cF$,
being $\ns{A} \NSneq \NSemptyset[\U]$,
by Proposition \ref{pro:propertiesofneutrosophicimagesandinverseimages}\,(1), we also have that
$\ns{f}\left(\ns{F}\right) \NSneq \NSemptyset[\V]$ and this means
that $\ns{f}\left(\cF\right)$ satisfies the condition (i)
of Definition \ref{def:singlevaluedneutrosophicfilterbase}.
Moreover, for every $\ns{G}_1, \ns{G}_2 \in \ns{f}\left(\cF\right)$,
there are some $\ns{F}_1, \ns{F}_2 \in \cF$ such that
$\ns{G}_1 \NSeq \ns{f}\left(\ns{F}_1\right)$ and $\ns{G}_2 \NSeq \ns{f}\left(\ns{F}_2\right)$.
Since $\cF$ is a SVN-filter base, there exists some $\ns{F}_3 \in \cF$ such that
$\ns{F}_3 \NSsubseteq \ns{F}_1 \NScap \ns{F}_2$.
Hence, said $\ns{G}_3 \NSeq \ns{f}\left(\ns{F}_3\right)$, we have that
$\ns{G}_3 \in \ns{f}\left(\cF\right)$, while
by Proposition \ref{pro:monotonicfneutrosophicimagesandinverseimages}\,(1) and
Proposition \ref{pro:propertiesofneutrosophicimagesandinverseimagesofgeneralizedunionandintersections}\,(2),
we obtain that
$\ns{G}_3 \NSeq \ns{f}\left(\ns{F}_3\right) \NSsubseteq
\ns{f}\left( \ns{F}_1 \NScap \ns{F}_2 \right) \NSsubseteq
\ns{f}\left( \ns{F}_1 \right) \NScap \ns{f}\left( \ns{F}_2 \right)
\NSeq \ns{G}_1 \NScap \ns{G}_2$.
This shows that $\ns{f}\left(\cF\right)$ satisfies also the condition (ii)
of Definition \ref{def:singlevaluedneutrosophicfilterbase}
and concludes our proof.
\end{proof}

\section{Single Valued Neutrosophic Ultrafilters}

In this section we consider the class of SVN-ultrafilters,
first proving that it is not empty and then establishing some characterizations
and properties that we think may be useful for further investigations.

\begin{definition}
\label{def:neutrosophicultrafilter}
Let $\cU$ be a SVN-filter base on $\SSVNS$, we say that it is a
\df{SVN-ultrafilter} if it is maximal in the partial ordered set $\left( \SVNF, \subseteq \right)$
of all SVN-filters over $\U$,
that is if there is no SVN-filter on $\SSVNS$ strictly finer than $\cU$,
or, equivalently, if any other SVN-filter containing $\cU$ coincides with $\cU$.
\end{definition}

\begin{proposition}
\label{pro:existenceofneutrosophicultrafilter}
Every SVN-filter base on $\SSVNS$ is contained in some SVN-ultrafilter over $\U$.
\end{proposition}
\begin{proof}
Let $\cF$ be a SVN-filter base over $\U$ and let us consider the set
$\mS = \left\{ \cG \in \SVNF : \, \cF \subseteq \cG \right\}$
of all SVN-filters on $\SSVNS$ containing $\cF$.
Obviously $\mS$ is a nonempty subset of the poset $\left( \SVNF, \subseteq \right)$
because $\completion{\cF} \in \mS$.
Now, let $\mC$ be a nonempty chain of $\mS$ and define
$\cM = \bigcup_{\cG \in \mC} \cG$ as the union of all SVN-filters of $\mC$.
Now, we have that $\cM$ is a SVN-filter over $\U$ because it satisfies all the conditions
of Definitions \ref{def:singlevaluedneutrosophicfilter}
and \ref{def:singlevaluedneutrosophicfilterbase}, that is:
\begin{enumr}
\item For every $\ns{G} \in \cM$, we have that there exists some $\cG \in \mC$ such that
$\ns{G} \in \cG$ and since $\cG$ is a SVN-filter, we immediately have that
$\ns{G} \NSneq \NSemptyset$.
\item For every $\ns{G}_1, \ns{G}_2 \in \cM$, there are some $\cG_1, \cG_2 \in \mC$
such that $\ns{G}_1 \in \cG_1$ and $\ns{G}_2 \in \cG_2$.
Since $\mC$ is a chain, i.e. a totally ordered subset of $\left( \mS, \subseteq \right)$,
$\cG_1$ and $\cG_2$ must be comparable and, without loss of generality, we can suppose that
$\cG_1 \subseteq \cG_2$.
Thus, $\ns{G}_1, \ns{G}_2 \in \cG_2$ and since $\cG_2$ is a SVN-filter, there exists some
$\ns{G}_3 \in \cG_2$ such that $\ns{G}_3 \NSsubseteq \ns{G}_1 \NScap \ns{G}_2$
with $\ns{G}_3 \in \cG_2 \subseteq \bigcup_{\cG \in \mC} \cG = \cM$.
\item For every $\ns{G} \in \cM$, and $\ns{A} \in \SSVNS$ such that $\ns{G} \NSsubseteq \ns{A}$,
there exists some $\cG \in \mC$ such that
$\ns{G} \in \cG$ and since $\cG$ is a SVN-filter, we immediately have that also $\ns{A} \in \cG$
and hence that $\ns{A} \in \cG \subseteq \bigcup_{\cG \in \mC} \cG = \cM$.
\end{enumr}
Moreover, being $\cF \subseteq \cG$, for every $\cG \in \mC$, we have that
$\cF \subseteq \bigcup_{\cG \in \mC} \cG = \cM$ and so that $\cM \in \mS$
is an upper bound for $\mC$.
Hence, by the Zorn's Lemma (see \cite{gemignani}), it follows that $\left( \mS, \subseteq \right)$ has a maximal element,
that is a SVN-ultrafilter $\cU$ containing $\cF$.
\end{proof}

\begin{proposition}
\label{pro:characterizationofneutrosophicultrafilters}
Let $\cU$ be a SVN-filter base on $\SSVNS$.
Then $\cU$ is a SVN-ultrafilter over $\U$ if and only if
for every $\nNS{A} \in \SSVNS$ which neutrosophically meets $\cU$ we have that $\ns{A} \in \cU$.
\end{proposition}
\begin{proof}
Suppose that $\cU$ is a SVN-ultrafilter and consider a \nameSVNS $\nNS{A} \in \SSVNS$
which neutrosophically meets $\cU$,
By Proposition \ref{pro:intersectionofaneutrosophicsetwithaneutrosophicfilterbase},
we have that $\cU \vee \ns{A}$ is a SVN-filter such that $\cU \subseteq \cU \vee \ns{A}$
but, being $\cU$ a SVN-ultrafilter, it must necessarily follow that $\cU \vee \ns{A} = \cU$
and so that $\ns{A} \in \cU$.
\\
Conversely, suppose that every \nameSVNS which neutrosophically meets the SVN-filter base $\cU$
belongs to $\cU$.
In order to prove first that $\cU$ is a SVN-filter over $\U$,
let $\ns{U} \in \cU$ and $\ns{A} \in \SSVNS$ such that $\ns{U} \NSsubseteq \ns{A}$.
We claim that $\ns{A}$ meets $\cU$.
In fact, for each $\ns{V} \in \cU$,
since $\cU$ is a SVN-filter base, we have that there exist some $\ns{W} \in \cU$ such that
$\ns{W} \NSsubseteq \ns{U} \NScap \ns{V} \NSsubseteq \ns{U} \NSsubseteq \ns{A}$
and, by Proposition \ref{pro:monotonic_neutrosophic_operators}\,(2),
we have that $\ns{W} \NScap \ns{V} \NSsubseteq \ns{A} \NScap \ns{V}$.
On the other hand, being also $\ns{W} \NSsubseteq \ns{U} \NScap \ns{V} \NSsubseteq \ns{V}$,
by Proposition \ref{pro:neutrosophicsubset_and_neutrosophicoperators}\,(1),
we have $\ns{W} \NScap \ns{V} \NSeq \ns{W}$
and hence that $\ns{W} \NSsubseteq \ns{A} \NScap \ns{V}$, with $\ns{W} \in \cU$.
Thus, by Definition \ref{def:singlevaluedneutrosophicfilterbase},
it follows that $\ns{A} \NScap \ns{V} \NSneq \NSemptyset$
and, by hypothesis, we obtain that $\ns{A} \in \cU$,
which proves that $\cU$ is a SVN-filter on $\SSVNS$.
\\
Moreover, in order to prove that $\cU$ is SVN-ultrafilter over $\U$,
suppose, by contradiction, that there is some SVN-filter $\cM$ over $\U$ such that $\cU \subset \cM$
and so that there exists some $\ns{M} \in \cM$ such that $\ns{M} \notin \cU$.
Hence, by  hypothesis, we have that $\ns{M}$ does not meet $\cU$, i.e. that
there exists some $\ns{U} \in \cU$ such that $\ns{U} \NScap \ns{M} \NSeq \NSemptyset$
but, being $\ns{M} \in \cM$ and $\ns{U} \in \cU \subset \cM$, this
contradicts the fact that $\cM$ is a SVN-filter and
concludes our proof.
\end{proof}

\begin{corollary}
\label{cor:svnultrafiltersbucomplement}
If $\cU$ is a SVN-ultraffilter on $\SSVNS$, then
for every $\nNS{A} \in \SSVNS$ it results $\ns{A} \in \cU$ or $\ns{A}^\NScompl \in \cU$.
\end{corollary}
\begin{proof}
Let $\nNS{A} \in \SSVNS$ an suppose, by contradiction, that $\ns{A} \notin \cU$ and $\ns{A}^\NScompl \notin \cU$.
By Proposition \ref{pro:characterizationofneutrosophicultrafilters}, we should have that
neither $\ns{A}$ does not neutrosophically meet $\cU$
nor $\ns{A}^\NScompl$ does not neutrosophically meet $\cU$, i.e. that there are some
$\ns{U}_1, \ns{U}_2 \in \cU$ such that $\ns{U}_1 \NScap \ns{A} \NSeq \NSemptyset$
and $\ns{U}_2 \NScap \ns{A}^\NScompl \NSeq \NSemptyset$.
Since $\cU$ is a SVN-filter base, it follows that there exists some
$\ns{U}_3 \in \cU$ such that $\ns{U}_3 \NSsubseteq \ns{U}_1 \NScap \ns{U}_2$.
Hence, by Proposition \ref{pro:unionandintersectiongeneralized}, we have that
$\ns{U}_3 \NSsubseteq \ns{U}_1$ and $\ns{U}_3 \NSsubseteq \ns{U}_2$
and, by Proposition \ref{pro:monotonic_neutrosophic_operators}\,(2), it also follows that
$\ns{U}_3 \NScap \ns{A} \NSsubseteq \ns{U}_1 \NScap \ns{A}$
and
$\ns{U}_3 \NScap \ns{A}^\NScompl \NSsubseteq \ns{U}_2 \NScap \ns{A}^\NScompl$.
Hence, $\ns{U}_3 \NScap \ns{A} \NSeq \NSemptyset$
and $\ns{U}_3 \NScap \ns{A}^\NScompl \NSeq \NSemptyset$.
Thus, by Propositions \ref{pro:propertiesunionandintersection}\, (4),
\ref{pro:absorptionneutrosophicsets} and \ref{pro:generalizeddistributive}\,(1),we have that
$\ns{U}_3 \NSeq \ns{U}_3 \NScap \ns{U}_3 \NSeq
\left( \ns{U}_3 \NScup \left( \ns{U}_3 \NScap \ns{A} \right) \right)
\NScap
\left( \ns{U}_3 \NScup \left( \ns{U}_3 \NScap \ns{A}^\NScompl \right) \right)
\NSeq
\ns{U}_3 \NScap \left( \left( \ns{U}_3 \NScap \ns{A} \right)
\NScup \left( \ns{U}_3 \NScap \ns{A}^\NScompl  \right)  \right)
\NSeq \ns{U}_3 \NScap \left( \NSemptyset \NScup \NSemptyset \right)
\NSeq \NSemptyset$
which is a contradiction to the fact that $\ns{U}_3 \in \cU$ and $\cU$ is a SVN-filter.
\end{proof}

\begin{remark}
In the classical filter's theory on crisp sets, the condition
of Corollary \ref{cor:svnultrafiltersbucomplement} is a ultrafilters characterization,
but in the case of filters on single valued neutrosophic sets the converse does not hold.
This is due to the fact that, as pointed out in Remark \ref{rem:intersectionandunionwithcomplement},
in general, the neutrosophic intersection of a \nameSVNS with its neutrosophic complement is not
the neutrosophic empty set, and it can be confirmed by the following example.
\end{remark}

\begin{example}
\label{rem:svnultrafilternotsatisfyingcondition}
Let $\U = \left\{ a,b,c \right\}$ be a finite universe set and
consider the SVN-principal filter $\cF = \completion{\ns{A}}
= \left\{ \ns{A}, \ns{B}, \ns{C}, \NSabsoluteset \right\}$
generated by $\ns{A}$,
where the \nameSVNS[s] $\nNS{A}, \nNS{B}, \nNS{C}$ and
$\NSabsoluteset = \NSbase{\und{1}}{\und{1}}{\und{0}}$ are
respectively defined by the following tabular representations:
\vspace{-3mm}
\begin{center}
\begin{tabular}{cc}
\begin{nstabular}{\U}{\ns{A}}{\M{A}}{\I{A}}{\NM{A}}
$a$ & 0.8 & 0.4 & 0
\\ \hline
$b$ & 0 & 0.1 & 0.9
\\ \hline
$b$ & 0 & 0 & 1
\end{nstabular}
\hspace{12mm} &
\begin{nstabular}{\U}{\ns{B}}{\M{B}}{\I{B}}{\NM{B}}
$a$ & 0.9 & 0.5 & 0
\\ \hline
$b$ & 0.8 & 0.6 & 0.1
\\ \hline
$c$ & 0 & 0.2 & 0.3
\end{nstabular}
\\[5mm]
\begin{nstabular}{\U}{\ns{C}}{\M{C}}{\I{C}}{\NM{C}}
$a$ & 1 & 0.5 & 0
\\ \hline
$b$ & 0 & 0.2 & 0.8
\\ \hline
$c$ & 0.7 & 0.6 & 0.5
\end{nstabular}
\hspace{12mm} &
\begin{nstabular}{\U}{\NSabsoluteset}{\M{\U}}{\I{\U}}{\NM{\U}}
$a$ & 1 & 1 & 0
\\ \hline
$b$ & 1 & 1 & 0
\\ \hline
$c$ & 1 & 1 & 0
\end{nstabular}
\end{tabular}
\end{center}
After observing that $\ns{A}\NSsubseteq\ns{B}$ and $\ns{A}\NSsubseteq\ns{C}$,
one can easily check that $\completion{\ns{A}}$ is a SVN-ultrafilter on $\SSVNS$.
However, if we consider the \nameSVNS $\nNS{Z}$ defined by the following tabular representation:
\vspace{-3mm}
\begin{center}
\begin{nstabular}{\U}{\ns{Z}}{\M{Z}}{\I{Z}}{\NM{Z}}
$a$ & 0 & 0 & 1
\\ \hline
$b$ & 0.7 & 0.3 & 0.5
\\ \hline
$c$ & 0.8 & 0.4 & 0.6
\end{nstabular}
\end{center}
it is a trivial matter to verify that neither $\ns{Z}$ nor its complement $\ns{Z}^\NScompl$\; belong to $\completion{\ns{A}}$.
\end{example}

\section{Conclusions and Perspectives}
In this paper we have introduced the notions of SVN-filter base, SVN-filter and SVN-ultrafilter
on the set $\SSVNS$ of all the single valued neutrosophic sets and we have investigated
some of their fundamental properties and relationships
concerning, in particular, the SVN-filter and SVN-filter base generated by some
neutrosophic filter subbase, the neutrosophic filter completion,
the principal SVN-filters, the infimum and the supremum of two SVN-filter bases,
the image of a SVN-filter base by a neutrosophic induced mapping between two universe sets,
the existence of SVN-ultrafilters and some their characterizations.
\\
We expect to continue the research on these topics by investigating supplementary
features and properties related to SVN-filters
and we hope that this comprehensive study will stimulate further developments
in the theory of Neutrosophic Sets
and will provide useful tools to demonstrate in an elegant and concise way
new properties for the class of Single Valued Neutrosophic Topological Spaces.
\\[6mm]
\noindent
\textbf{Acknowledgements}
The authors would like to express their sincere gratitude to the Anonymous Referees
for their valuable suggestions and comments which were precious in improving of the paper.


%

\end{document}